\documentclass[11pt]{article}
\usepackage{amsthm}
\usepackage{amssymb}
\usepackage{epsfig}
\usepackage{amsthm}
\usepackage{srcltx}
\usepackage[dvipsnames,usenames]{color}
\usepackage{paralist}
\usepackage{amsmath}
\usepackage{amsfonts}

\newcounter{contador}
\newcounter{teoA}

\newtheorem{propo}[contador]{Proposition}
\newtheorem{teo}[contador]{Theorem}
\newtheorem{lem}[contador]{Lemma}

\theoremstyle{remark}
\newtheorem{nota}[contador]{Remark}

\newcounter{ex}

\newcommand{\C}{{\mathbb C}}

\title{A dynamic Parrondo's paradox for continuous\\ seasonal systems\footnote{The authors are supported by
Ministry of Economy, Industry and Competitiveness--State Research Agency
of the Spanish
Government through grants MTM2016-77278-P  (MINECO/AEI/FEDER, UE, first and second
authors) and DPI2016-77407-P
 (MINECO/AEI/FEDER, UE, third author). The first and second authors are also supported by the grant 2017-SGR-1617  from
AGAUR,  Generalitat de Catalunya. The third author acknowledges the
group's research recognition 2017-SGR-388 from AGAUR, Generalitat de
Catalunya.}}

\author{Anna Cima$^{(1)}$, Armengol Gasull$^{(1)}$ and V\'{\i}ctor Ma\~{n}osa$^{(2)}$
  \\*[.1truecm]
{\small \textsl{$^{(1)}$ Departament de Matem\`{a}tiques, Facultat
de Ci\`{e}ncies,}}
\\*[-.25truecm] {\small \textsl{Universitat Aut\`{o}noma de Barcelona,}}
\\*[-.25truecm] {\small \textsl{08193 Bellaterra, Barcelona, Spain}}
\\*[-.25truecm] {\small \textsl{cima@mat.uab.cat,
gasull@mat.uab.cat}}\\
\\*[-.25truecm] {\small \textsl{$^{(2)}$ Departament de Matem\`{a}tiques,}}
\\*[-.25truecm] {\small \textsl{Universitat Polit\`{e}cnica de Catalunya}}
\\*[-.25truecm] {\small \textsl{Colom 11, 08222 Terrassa, Spain}}
\\*[-.25truecm] {\small \textsl{victor.manosa@upc.edu}}}
\begin{document}

\maketitle
\begin{abstract}  We show that planar continuous alternating systems, which can be used to model
 systems with seasonality, can exhibit a type of Parrondo's dynamic paradox, in which the
  stability of an equilibrium, common to all seasons is reversed for the global seasonal system.
\end{abstract}



\newpage

\section{Introduction and Main results}

For dynamical systems given by differential equations, alternating systems take the form
\begin{equation}\label{e:eq1a}
\begin{cases}\begin{array}{l}
\dot {x_1}(t)=X_1(\mathbf{x}(t))\,\mbox{ for } t \,\mbox{ such that
} t \,(\mathrm{mod}\, T)\,
\in [0,T_1), \\
\dot{{x_2}}(t)=X_2(\mathbf{x}(t))\,\mbox{ for } t \,\mbox{ such that
} t \,(\mathrm{mod}\, T)\, \in [T_1,T_1+T_2),\\
\vdots \\
\dot{{x_n}}(t)=X_n(\mathbf{x}(t))\,\mbox{ for } t \,\mbox{ such that
} t \,(\mathrm{mod}\, T)\, \in [T_1+\cdots +T_{n-1},T_1+\cdots
+T_n),
\end{array}\end{cases}
\end{equation}
where $T=\sum_{j=1}^{n}T_j$, with $T_j>0$ for $j=1,2,\ldots,n,$ and $X_j$  being class
$\mathcal{C}^1$ vector fields. They can be used to model continuous
seasonal systems with $n$ seasons of durations $T_1, T_2,\ldots
,T_n$. It is not necessary to recall the importance of these kind of
systems in mathematical biology, for instance in population models
for which the seasonality has an effect in  the reproduction and
mortality rates due to environmental circumstances or to  human
intervention like harvesting, see \cite{HZ,X,XBD} and references
therein  (and \cite{CHL14,F,KS,Liz17} for discrete examples); or
also in epidemiological models with periodic contact rate, see
\cite{BCO} and the references therein.

The so called \emph{Parrondo's paradox} is a paradox in game theory,
that in a few words says that {\it a combination of losing strategies
  can  become a winning strategy,}  see~\cite{HA,Par}. Several  dynamical versions of related paradoxes are presented
in \cite{CLP,CGM12,CGM13,CGM18} for discrete non-autonomous
dynamical systems. In the first paper  the authors combine
periodically one-dimensional maps $f_1$ and $f_2$  to  give rise to
chaos or order. The existence of discrete systems that exhibit
(numerically) chaotic dynamics by alternating regular, or more
precisely, integrable systems, has been referred in \cite{CGM12} and
\cite{CGM13}. In this last reference also are shown
 alternating systems with regular (integrable) dynamics obtained by alternating an integrable map and a numerically chaotic one. In
 \cite{CGM18} we  study a local problem, but in any dimension.  In particular, we
relate the stability of a common fixed point of two planar maps,
$F_1$ and $F_2,$ with the stability of this point for $F_2\circ
F_1.$ We prove that in the non-hyperbolic case, with complex
conjugated eigenvalues (elliptic fixed points), a common attracting
character of the common fixed point of $F_1$ and $F_2$, can be
reversed for $F_2\circ F_1.$ This phenomenon is the one that we
named  \emph{ Parrondo's dynamic
    type paradox} for 2-periodic discrete dynamical systems. In this
    work we will show that a similar dynamical paradox appears for continuous seasonal systems.

As noted in \cite{BCO}, the asymptotic  stability of the equilibria
of a  seasonal system, for instance the disease-free equilibrium of
an epidemiological model, is a more complex  issue than in the
autonomous case. In this note we evidence that a seasonal system of
type \eqref{e:eq1a}  can exhibit a dynamic-type Parrondo's paradox,
in which the stability of an equilibrium common to all stations
(either locally asymptotically stable, LAS from now on, or a
repeller), is reversed for the seasonal system \eqref{e:eq1a}. That
is, we show that there exist systems \eqref{e:eq1a} with a common
singular point which is LAS (resp. repeller) for each season system
$\dot{\mathbf{x}}=X_i(\mathbf{x})$ for $i=1,\ldots,n$ and such it is
a repeller (resp. LAS) for the global seasonal system.

To simplify the problem, we prove the existence of the
Parrondo's-type paradox  for planar differential with two seasons,
both with duration $T_1=T_2=1$. Hence  systems of the form
\begin{equation}\label{e:eq1}
\begin{cases}\begin{array}{l}
\dot{\mathbf{x}}(t)=X_1(\mathbf{x}(t))\,\mbox{ for } t\in [2k,2k+1), \\
\dot{\mathbf{x}}(t)=X_2(\mathbf{x}(t))\,\mbox{ for } t\in
[2k+1,2k+2),\, k\in\mathbb{N}\cup\{0\},
\end{array}\end{cases}
\end{equation}
with $\mathbf{x}(t)\in\mathbb{R}^2$. Our main result is:

\begin{teo}\label{t:main}
There exist planar polynomial vector fields $X_1$ and $X_2$ sharing
a  common singular point which is LAS (resp. repeller) for both of
their associated differential systems, and such that it is a
repeller (resp. LAS) for the $2$-seasonal differential system
\eqref{e:eq1}.
\end{teo}

Notice that this theoretical result opens a practical interesting
situation.    Let us consider a system where the state variables
represent the density of individuals of an age-structured population
of a species that can be potentially dangerous to humans, like for
instance mosquitoes, \cite{JS}. Let us assume that for two different
environmental situations (the two seasons) the zero solution is a
repeller. Of course, this corresponds to unwanted scenarios since,
in each season, for an arbitrary small initial density of
individuals the amount of them increases over time. Then, it might
happen that alternating both situations we get a system with the
origin as a LAS critical point, implying the population decline (and
long term extinction) of the dangerous species.

In the following, we will use complex notation in order to simplify the expressions.
Hence instead of taking   planar vector fields $U(x,y)\partial/\partial x+V(x,y)\partial/\partial y$ with $(x,y)\in \mathbb{R}^2$, we will consider
the same vector fields but in complex notation
$X(z,\bar{z})=F(z,\bar z)\partial/\partial z$ where $z=x+iy\in\C$ , with associated differential equation $\dot{z}=X(z,\bar{z})$.

One of the key ingredients in our approach will be to know whether for a given  local polynomial
diffeomorphism   of the form
\begin{equation}\label{e:F}
F(z,\bar{z})=\mathrm{e}^{i\alpha}z+\sum\limits_{j+k=2}^{n} f_{j,k}z^j\bar{z}^k, \quad \alpha\in(0,2\pi),
\end{equation}
of degree at most $n,$ that has  a non-hyperbolic elliptic fixed
point at the origin, there exists of  a polynomial vector field of
type
\begin{equation}\label{e:X}
X(z,\bar{z})=i\alpha z+\sum_{j+k=2}^{n} a_{j,k} z^j\bar{z}^k
\end{equation}
and such that its associated flow $\varphi(t;z,\bar{z})$ satisfies
\begin{equation}\label{e:temps1}
\varphi(1;z,\bar{z})=F(z,\bar{z})+O(n+1),
\end{equation}
for every $(z,\bar z),$ for $z$ in a small enough  neighborhood of
$z=0.$ As we will see, for our purposes we only will need to
consider the cases $n=2$ or  $n=3.$ This question is solved in next
section.

We also would like to comment that very few planar  polynomial maps
are exactly a flow at a fixed time, {\it i.e.} the remainder term
$O(n+1)$ in~\eqref{e:temps1} is identically zero. They are the
so-called {\it polynomial flows}, and the normal forms of their
corresponding vector fields are given in \cite[Thm. 4.3]{BM}.

In fact, ultimately, the proof of Theorem \ref{t:main} relies on the
fact that, near a critical point, the flow of some suitable vector
fields are such, up to certain  fixed order on the initial
conditions, their associated time-$1$ maps are the ones given in
Example 7 of \cite{CGM18}. We recall them in Proposition
\ref{p:maps}. These maps display the features of the Parrondo's
dynamic paradox for the dynamics induced by iterating maps and this
fact translates to alternating systems of differential equations.
This proof   is given in Section~\ref{s:proof}.

As a byproduct of our study we obtain the following result that we believe is interesting by itself. Its proof is given in Section~\ref{t:2}.

\begin{teo}\label{p:minimal} It holds:
\begin{enumerate}[(i)]

\item  Consider a local diffemorphism of the form~\eqref{e:F}, where $\mathrm{e}^{i\alpha}$ is not a
    root of the  unity. Then, for any $n\geq 2$  there is a unique
    polynomial vector field of the form~\eqref{e:X} and degree at most $n$  such that its flow satisfies Equation~\eqref{e:temps1}.

\item   For any $n\geq 2$, there exists a map $F$ of
 the form~\eqref{e:F} with $\alpha=2\,\pi/(n+1)$ for which there is no $\mathcal{C}^{n+1}$ vector field
  whose flow satisfies Equation \eqref{e:temps1}.
\end{enumerate}
\end{teo}

Notice that the above result implies the existence of planar
polynomial local diffeomorphisms, preserving orientation, that can
not be given as the flow at a fixed time of  smooth planar vector
fields.  Particular examples of such maps are given in
\eqref{e:Fcontaexemple}.

\section{Vector fields with prescribed maps as time-$1$ map}

\subsection{A recurrent procedure}\label{s:recu}

First we establish the structure equation that must satisfy the
first terms of a flow map associated with a vector field.  It is
easy to prove that if a flow map satisfies Equations
\eqref{e:F}--\eqref{e:temps1}, then the vector field must have the
form $X(z,\bar{z})= i \alpha z+O(2),$ so we must work with vector
fields with this fixed linear part. If we impose that $X$ is
polynomial of degree~$n$ we can write $X(z,\bar{z})=i\alpha
z+\sum_{j+k=2}^{n} a_{j,k} z^j\bar{z}^k.$ When we only assume that
it is of class $\mathcal{C}^{n+1},$ near the origin we can write it
as $X(z,\bar{z})=i\alpha z+\sum_{j+k=2}^{n} a_{j,k}
z^j\bar{z}^k+O(n+1).$ In any case, by plugging the Taylor expansion
of $\varphi(t;z,\bar{z})$ in the expression of the differential
system $\dot{z}=X(z,\bar{z})$, that is by imposing
$d\varphi(t;z,\bar{z})/dt=X(\varphi(t),\overline{\varphi(t)})=i\alpha
\varphi(t;z,\bar{z})+\sum_{j+k=2}^{n} a_{j,k}
\varphi^j(t,z)\bar{\varphi}^k(t,z)+O(n+1),$ and from  a power
comparison argument we get the following result:

\begin{lem}\label{l:lemequaciolineal}
Let $X(z,\bar{z})=i\alpha z+\sum_{j+k=2}^{n} a_{j,k}
z^j\bar{z}^k+O(n+1)$ with $\alpha\in(0,2\pi)$ be
 a planar $\mathcal{C}^{n+1}$ vector field. Then, in a neighborhood of the origin, its flow is given by
$ \varphi(t;z,\bar{z})=\mathrm{e}^{i\alpha t}z+ \sum_{j+k= 2}^n
\varphi_{j,k}(t)z^j\bar{z}^k+O(n+1),$ where for each $j,k\in
\mathbb{N}$ such that $2\leq j+k\leq n$ the functions
$\varphi_{j,k}(t)$ satisfy the linear differential equation
\begin{equation}\label{e:edolineal}
\dot{\varphi}_{j,k}(t)=i\alpha\, \varphi_{j,k}(t)+a_{j,k}
\mathrm{e}^{i(j-k)\alpha t}+b_{j,k}(t)\,\,\,\text{with}\,\,\,
\varphi_{j,k}(0)=0,
\end{equation}
where $b_{j,k}(t)=\sum_{\gamma\in S_{j,k}} P_\gamma(t)
\mathrm{e}^{\gamma i t}$ and $S_{j,k}\subset \mathbb{Z}$ is a finite
set,  $P_\gamma$ depends on the values on the coefficients
$a_{\ell,m}$ and the functions $\varphi_{\ell,m}(t)$ with $2\leq
\ell+m<j+k$.
\end{lem}

By using the above result, given a  map \eqref{e:F}, we want either
to obtain a planar polynomial vector field  $X(z,\bar{z})$
 such
that in a neighborhood of the origin its flow satisfies
\eqref{e:temps1} or to prove that there is no $\mathcal{C}^{n+1}$
vector field which flow satisfies \eqref{e:temps1}. We do it by a
recursive procedure.

Indeed, suppose that we have computed the coefficients of $X$ up to
order $\kappa-1$ for $2< \kappa\leq n$. To compute any coefficient
$a_{j,k}$ with $j+k=\kappa$, we solve the initial value problem
\eqref{e:edolineal} and impose Equation \eqref{e:temps1}.  If
$j-k-1= 0$, then
$$
\varphi_{j,k}(1)=\mathrm{e}^{i\alpha }\left[ a_{j,k}+\int_0^1 b_{j,k}(\tau)\mathrm{e}^{-i\alpha \tau}d\tau \right]=f_{j,k}.
$$
In this case we can isolate the coefficient $a_{j,k}$, thus contributing to determinate the expression of the vector field.

If $j-k-1\neq 0$, then we have
\begin{equation}\label{e:E}
\varphi_{j,k}(1)=\mathrm{e}^{i\alpha }\left[\frac{a_{j,k}}{i(j-k-1)\alpha}\left(\mathrm{e}^{i(j-k-1)\alpha }
-1\right)+\int_0^1 b_{j,k}(\tau)\mathrm{e}^{-i\alpha \tau}d\tau \right]=f_{j,k}.
\end{equation}
From the above equation we always can isolate the coefficient $a_{j,k}$ except in the case that
$$
\mathrm{e}^{i(j-k-1)\alpha }-1=0,
$$ or, in other words if  $\mathrm{e}^{i\alpha}$ is a $|j-k-1|$-root of the unity.
 In this case we say that there appears a \emph{resonance} associated with the coefficient
  $a_{j,k}$, and the equation \eqref{e:E} is satisfied for every value of $a_{j,k}$
   (thus leading to a parametric family of vector fields) if and only if it is satisfied
   \emph{the compatibility equation} corresponding to the coefficient $a_{j,k}$:
\begin{equation}\label{e:compatibilitat}
\mathrm{e}^{i\alpha}\,\int_0^1 b_{j,k}(\tau)\mathrm{e}^{-i\alpha \tau}d\tau=f_{j,k}.
\end{equation}
Otherwise, we get an obstruction for $F$ to be the time-$1$  map of
a polynomial  (or $\mathcal{C}^{n+1}$)  vector field, see the proof
of  Theorem \ref{p:minimal} for examples of  polynomial  maps for
which there is no vector field whose flow  satisfies
\eqref{e:temps1}.

In fact, observe that if for any couple $\ell$ and $m$ with $2\leq
\ell+m<j+k$ there is not a resonance, then the function
$b_{j,k}(t)=\sum_{\gamma\in S_{j,k}} P_\gamma(t) \mathrm{e}^{\gamma
i t}$ introduced in Lemma \ref{l:lemequaciolineal}, depends on the
values of the previous coefficients $a_{\ell,m}$, thus on the
previous coefficients $f_{\ell,m}$.   On the contrary, if there
exists a couple of values $\ell$ and $m$ with $2\leq \ell+m<j+k$
giving rise to a resonance (that is, $\mathrm{e}^{i\alpha}$ is a
$|\ell-m-1|$-root of the unity) and the compatibility condition
associated with $a_{\ell,m}$ is satisfied, then the function
$b_{j,k}(t)$ also depends on the parameter $a_{\ell,m}$.

Also observe that a resonance may appear at different order levels,
so that in order to  obtain the associated vector field, we must
identify the first order in which a resonance appears and verify
that  each compatibility equation is fulfilled. In that case, we
can proceed by solving the different equations \eqref{e:edolineal}
for higher orders, by carrying the expressions of the indeterminate
terms, and verifying that the next different compatibility equations
are also satisfied.
\begin{nota}\label{resnivelln}
Fixing an order $n$, if we consider the pairs $(j,k)$ with $j+k=n$
we get that $|j-k-1|\in\{0,2,4,\ldots,n+1\}$ if $n$ is odd and
$|j-k-1|\in\{1,3,5,\ldots,n+1\}$ if $n$ is even. Hence, if a
resonance appears at order $n$ and it has not appeared at order
$k<n,$ then $\mathrm{e}^{i\alpha}$ is an $m$-root of unity with
$m\in\{2,4,\ldots,n+1\}\,\,\text{if}\,\,n\,\,\text{is odd}$ and
$m\in\{3,5,\ldots,n+1\}\,\,\text{if}\,\,n\,\,\text{is even}.$
\end{nota}

Summarizing the above recursive procedure we obtain the following
result:

\begin{teo}\label{p:compatibilitat}
Consider a polynomial map $F$ of degree   $n$ of the form
\eqref{e:F}.
\begin{itemize}[(i)]
\item If $\mathrm{e}^{i\alpha}$ is not a $|j-k-1|$-root of the unity
for all couple $j,k$ with $j+k\in\{0,1,\dots ,n\}$ then there exists
a unique polynomial vector field of degree at most $n$ such that its
associated flow  satisfies
$\varphi(1;z,\bar{z})=F(z,\bar{z})+O(n+1).$

\item If $\mathrm{e}^{i\alpha}$ is a $|j-k-1|$-root of the unity
for certain $j,k$ with $j+k\in\{0,1,\dots ,n\}$ and the
compatibility equation \eqref{e:compatibilitat} corresponding to the
coefficient $a_{j,k}$ is not satisfied, then there is no
$\mathcal{C}^{n+1}$  vector fiel such that its associated flow
satisfies $\varphi(1;z,\bar{z})=F(z,\bar{z})+O(n+1).$

\item If there are $\ell$
couples $j,k$ with $j+k\in\{0,1,\dots ,n\}$ such that
$\mathrm{e}^{i\alpha}$ is a $|j-k-1|$-root of the unity and the
compatibility equations \eqref{e:compatibilitat} corresponding to
the coefficients $a_{j,k}$ are satisfied, then there exists an
$\ell$-parametric family of polynomial vector fields of degree at
most $n$ satisfying $\varphi(1;z,\bar{z})=F(z,\bar{z})+O(n+1).$
\end{itemize}
\end{teo}

In the next sections we present the explicit expressions for the
vector  fields associated with quadratic and cubic maps of the form
\eqref{e:F}, satisfying Equation \eqref{e:temps1} for $n=2$ and
$n=3$ respectively.

\subsection{Vector fields for quadratic maps}

 The whole scene in the quadratic case is described in the next
proposition. Observe that
 in the above scheme, at order two a resonance can only occur
if $\omega=\mathrm{e}^{i\alpha}$ is a cubic root of unity. This can
be seen by taking the function $r(j,k)=|j-k-1|$ and observing that
it takes the values $r(2,0)=r(1,1)=1$ and $r(0,2)=3$.

\begin{propo}\label{p:propoquadratics}
Set $F(z,\bar{z})=\omega z+\sum_{j+k=2}f_{j,k}z^j\bar{z}^k$, where
$\omega=\mathrm{e}^{i\alpha}$ with $\alpha\in(0,2\pi)$. Then
\begin{enumerate}[(a)]
\item If $\omega$ is not a cubic root of unity, then there exists a unique quadratic vector field satisfying \eqref{e:temps1} with $n=2$,
given by $X(z,\bar{z})=i\alpha z+\sum_{j+k=2} a_{j,k} z^j\bar{z}^k$  where
\begin{equation}\label{e:coefsquadratics}
a_{2,0}={\frac {i\alpha\,f_{{2,0}}}{\omega\, \left( \omega-1 \right)
}},\quad a_{1,1}={\frac {i\alpha\,f_{{1,1}}}{\omega-1}},\quad
a_{0,2}={\frac {3i\alpha\,{\omega}^{2}f_{{0,2}}}{\omega^3-1
  }}.
\end{equation}
\item  Assume that $\omega$ is a cubic root of unity. If $f_{0,2}=0$, then
 there exists an one-parameter family of quadratic vector fields satisfying \eqref{e:temps1} with
$n=2.$ In this case the coefficients $a_{2,0}$ and $a_{1,1}$ of such
a vector field are the ones given in Equation
\eqref{e:coefsquadratics} and $a_{0,2}$ is the free parameter. If
$f_{0,2}\ne 0$ then there is no $\mathcal{C}^3$ vector field
satisfying \eqref{e:temps1} with $n=2.$
\end{enumerate}
\end{propo}

\begin{proof} Consider a quadratic map $F(z,\bar{z})=i\alpha z+f_{2,0}z^2+f_{1,1}z\bar{z}+f_{0,2}\bar{z}^2$ and a vector field of the form $
 X(z,\bar{z})=i\alpha z+a_{2,0}z^2+a_{1,1}z\bar{z}+a_{0,2}\bar{z}^2$.
If we search for its associated flow $
\varphi(t;z,\bar{z})=\mathrm{e}^{i\alpha t}z+
\varphi_{2,0}(t)z^2+\varphi_{1,1}(t)z\bar{z}+\varphi_{0,2}(t)\bar{z}^2+O(3)$,
by plugging this expression in the differential equation
$\dot{z}=X(z,\bar{z})$, we get the equations
\eqref{e:varphiquadratsprima}:
\begin{align}
&\dot{\varphi}_{2,0}=i\alpha \varphi_{2,0}+a_{2,0}\mathrm{e}^{2i\alpha t},\nonumber\\
&\dot{\varphi}_{1,1}=i\alpha \varphi_{1,1}+a_{1,1},\label{e:varphiquadratsprima}\\
&\dot{\varphi}_{0,2}=i\alpha \varphi_{0,2}+a_{0,2}\mathrm{e}^{-2i\alpha t}.\nonumber
\end{align}
with the conditions ${\varphi}_{2,0}(0)=0,{\varphi}_{1,1}(0)=0 $ and
${\varphi}_{0,2}(0)=0.$ By integrating these equations, evaluating
their solutions at time $t=1$ and imposing
$\varphi(1;z,\bar{z})=F(z,\bar{z})+O(3)$ we get the corresponding
equations \eqref{e:edolineal}:
\begin{align}
&\varphi_{2,0}(1)=\frac{i}{\alpha} a_{2,0}\left(1-\mathrm{e}^{i\alpha}\right)\mathrm{e}^{i\alpha}=f_{2,0},\nonumber\\
&\varphi_{1,1}(1)=\frac{i}{\alpha} a_{1,1}\left(1-\mathrm{e}^{i\alpha}\right)=f_{1,1},\label{e:varphiquadrats}\\
&\varphi_{0,2}(1)=-\frac{i}{3\alpha}
a_{0,2}\left(1-\mathrm{e}^{-3i\alpha}\right)\mathrm{e}^{i\alpha}=f_{0,2}.\nonumber
\end{align}
Since $\alpha\in(0,2\pi)$, the first two equations can always be
solved giving the  values for $a_{2,0}$ and $a_{1,1}$ in Equation
\eqref{e:coefsquadratics}. The third equation fixes a value of
$a_{0,2}$ unless $\omega$ is a third root of unity, obtaining the
expressions in \eqref{e:coefsquadratics}, thus proving $(a).$

If $\omega$ is a cubic root of the unity, then the  compatibility
condition \eqref{e:compatibilitat} associated with the coefficient
$a_{0,2}$ is $ f_{0,2}=0, $ and the result in statement $(b)$
follows.~\end{proof}

\subsection{Vector fields for cubic maps}

Given a cubic map, to search a cubic vector field satisfying
\eqref{e:temps1} with $n=3$, first we notice that the resonances
only occur when $\mathrm{e}^{i\alpha}$ is a third root of the unity,
when is a square root of the unity, or when is a primitive fourth
root of the unity, see Remark~\ref{resnivelln}. According to Theorem
\ref{p:compatibilitat}, if $\mathrm{e}^{i\alpha}$ is not such a root
of the unity there exists a unique polynomial vector field
satisfying \eqref{e:temps1}.

Also according to Theorem \ref{p:compatibilitat},  if
$\mathrm{e}^{i\alpha}$ is a third root of the unity and the
compatibility condition associated to $a_{0,2}$ is satisfied, then
there exists an one-parameter family of vector fields satisfying
\eqref{e:temps1}.  If $\mathrm{e}^{i\alpha}$ is a square root of the
unity (hence it also is a fourth-root of unity) and the
compatibility condition associated with $a_{3,0}$, $a_{1,2}$ and
$a_{0,3}$ are fulfilled, then there exists a three-parametric family
of such vector fields. And finally, if $\mathrm{e}^{i\alpha}$ is a
primitive quartic root of the unity and the compatibility condition
associated with $a_{0,3}$  holds, then there exists an one-parameter
family of such vector fields. All these four cases are studied in
the next four propositions:

\begin{propo}\label{p:propocubics-nores}
Set $F(z,\bar{z})=\omega z+\sum_{j+k=2}^3f_{j,k}z^j\bar{z}^k$, where
$\omega=\mathrm{e}^{i\alpha}$ with $\alpha\in(0,2\pi)$. If $\omega$
is not a quadratic, cubic or fourth root of unity, then there exists
a unique cubic vector field satisfying \eqref{e:temps1} with $n=3$,
$X(z,\bar{z})=i\alpha z+\sum_{j+k=2}^3 a_{j,k} z^j\bar{z}^k$, where
the coefficients $a_{2,0}$, $a_{1,1}$ and $a_{0,2}$ are the ones
given in \eqref{e:coefsquadratics}, and
$$a_{3,0}={\frac {-i\alpha\, P_{3,0} }{{\omega}^{2} \left(
\omega^3-1 \right)  \left( \omega+1 \right)}},$$ with
$$P_{3,0}=
\left(
\overline{f_{{0,2}}}f_{{1,1}}-2\,f_{{3,0}} \right) {\omega}^{3}+
 2\left( \overline{f_{{0,2}}}f_{{1,1}}+{f^{2}_{{2,0}}}-f_{{3,0
}} \right) {\omega}^{2}+ 2\left( {f^{2}_{{2,0}}}-f_{{3,0}}
 \right) \omega+2\,{f^{2}_{{2,0}}};
$$

$$a_{2,1}=\frac{-i \, P_{2,1}}
{{\omega}^{2} \left( \omega^3-1 \right) ^{2} },
$$
with
\begin{align*}
P_{2,1}=& \left(  \left( i+\alpha \right)   \left| f_{
{1,1}} \right|^{2}+if_{{2,1}} \right) {\omega}^{7}+ \left(
 \left( i+2\,\alpha \right)  \left| f_{{1,1}} \right|^{2}-2\,if_{{1,1}}f_{{2,0}} \right) {\omega}^{6}+\\
 &
\left(
 \left( i+3\,\alpha \right)  \left| f_{{1,1}} \right|^{2}+ \left( 2\,i+6\,\alpha \right)  \left| f_{{0,2}}
 \right|^{2}-f_{{1,1}}f_{{2,0}} \left( i+\alpha \right)
 \right) {\omega}^{5}+\\
 & \left(  \left( -i+2\,\alpha \right)
 \left| f_{{1,1}} \right|^{2}- \left( i+2\,\alpha \right) f_
{{2,0}}f_{{1,1}}-2\,if_{{2,1}} \right) {\omega}^{4}+ \left( 3\,f_{{1,1
}}f_{{2,0}}- \left| f_{{1,1}} \right|^{2} \right)
\times\\
& \left( i-\alpha \right) {\omega}^{3}+ \left( -i  \left| f_{{1,
1}} \right|^{2}-2\,i \left| f_{{0,2}} \right|^{2}+\left( i-2\,\alpha \right) f_{{1,1}} f_{{2,0}} \right) {
\omega}^{2}+\\
& \left( f_{{2,0}}f_{{1,1}} \left( i-\alpha \right) +if_{{2
,1}} \right) \omega-if_{{1,1}}f_{{2,0}};
\end{align*}

$$a_{1,2}={\frac{-i\alpha\,P_{1,2}} {
 \left( \omega^3-1 \right)  \left( \omega+1 \right)   }},$$ with
\begin{align*}
P_{1,2}=& f_{{1,1}}\overline{f_{{2,0}}}\,{\omega}^{4}+
 \left( 2\,\overline{f_{{1,1}}}f_{{0,2}}+\overline{f_{{2,0}}}f_{{1,1}}
-2\,f_{{1,2}} \right) {\omega}^{3}+ \left( 4\,\overline{f_{{1,1}}}f_{{0
,2}}+\overline{f_{{2,0}}}f_{{1,1}}+4\,f_{{0,2}}f_{{2,0}}+\right.\\
&
\left.{f^{2}_{{1,1}}}-2\,f_{{1,2}} \right) {\omega}^{2}+ \left( 2\,f_{{0,2}}f_{{2,0}}+{f^{2}_{{1,1}}}-2\,f_{{1,2}} \right) \omega+{f^{2}_{{1,1}}};
\end{align*}
and
$$
a_{0,3}={\frac{-i\alpha\,{\omega}^{2}\,  P_{0,3}} { \left(
\omega^2+1 \right)   \left( \omega^3-1
 \right)  \left( \omega+1 \right) }},
$$ with
\begin{align*}
P_{0,3}=&
 2f_{{0,2}}\overline{f_{{2,0}}}\,{\omega}^{4}+
 4\,\left( f_{{0,2}}\overline{f_{{2,0}}} -f_{{0,3}
} \right) {\omega}^{3}+ \left( 6\, f_{{0,2}}\overline{f_{{2,0}}}+3\,f_{{0,2}}f_{{1,1}}-4\,f_{{0,3}} \right) {\omega}^{2}+\\
&
 2\,\left( f_{{0,2}}f_{{1,1}}-2\,f_{{0,3}} \right) \omega+f_{{0,2}}f_{
{1,1}}.
\end{align*}
\end{propo}
\begin{proof}
Consider the cubic map $F(z,\bar{z})$ and a cubic vector field $
 X(z,\bar{z})$.
To search for the flow $
\varphi(t;z,\bar{z})=\mathrm{e}^{i\alpha}z+\sum_{j+k=2}^3\varphi_{j,k}(t)z^j\bar{z}^k+O(4)$
associated with $X(z,\bar{z})$, we plug this expression in the
differential equation $\dot{z}=X(z,\bar{z})$, and we get the
corresponding equations~\eqref{e:edolineal}. The equations
corresponding to the quadratic terms are the ones obtained in the
proof of Proposition \ref{p:propoquadratics}, that is Equations
\eqref{e:varphiquadratsprima}--\eqref{e:varphiquadrats}, thus we
obtain the same terms for the quadratic terms of the vector field.
To obtain the cubic terms we follow the same procedure. For reasons
of space we omit the steps to obtain the expression of all the four
equations \eqref{e:edolineal} and its corresponding solutions. We
only show the case of the equation corresponding to the coefficient
$a_{3,0}$. Indeed, we get:
$$
\dot{\varphi}_{3,0}(t)=i\alpha\,\varphi_{3,0}(t)+a_{3,0}\mathrm{e}^{3i\alpha\,t}+
\frac {i\alpha\, Q_{3,0}(t)}{{\omega}^{2} \left( \omega-1 \right)
^{2} \left( { \omega}^{2}+\omega+1 \right) },
$$
where
$$
Q_{3,0}(t)= -2\left( {\omega}^{2}+\omega+1 \right)\,{f^{2}_{{2,0}}}
{{\rm e}^{2\, i\alpha\,t}}+ \left(
\overline{f_{{0,2}}}f_{{1,1}}{\omega}^{3}+2 \, \left(
{\omega}^{2}+\omega+1 \right)\,f^{2}_{2,0}\right) {{\rm
e}^{3\,i\alpha\,t}}-\overline{f_{{0,2}}}f_{{1,1}}{\omega}^{3}.
$$
By integrating this differential equation and imposing Equation \eqref{e:temps1},
we obtain the corresponding Equation \eqref{e:E}:
$$
{\frac {-i\left( \omega^2-1 \right) \omega\,
a_{{3,0}}}{2\alpha}}+{\frac {\overline{f_{{0,2}}}f_{{1,1}}{\omega}^{3}+ \left( 2\,
\overline{f_{{0,2}}}f_{{1,1}}+2\,{f^{2}_{{2,0}}} \right) {\omega}^{2}+
2\,{f^{2}_{{2,0}}}\,\omega\,+2\,{f^{2}_{{2,0}}}
}{2\omega\,
 \left( {\omega}^{2}+\omega+1 \right) }}=f_{{3,0}},
$$
thus we get the expression of the coefficient $a_{3,0}$ in the
statement.  The other expressions are obtained in a similar
way.~\end{proof}

\begin{propo}
Set $F(z,\bar{z})=\omega z+\sum_{j+k=2}^3f_{j,k}z^j\bar{z}^k,$ where
$\omega=\mathrm{e}^{i\alpha}$ is a primitive third root of unity.
Then there exists a  cubic vector field satisfying~\eqref{e:temps1}
for $n=3$ if and only if $f_{0,2}=0$. In this case, there is an
one-parameter family of cubic vector fields
satisfying~\eqref{e:temps1} for $n=3,$ whose coefficients $a_{2,0}$
and $a_{1,1}$ are the ones given in Equation
\eqref{e:coefsquadratics}, $a_{0,2}$ is a free parameter, and
$$ a_{3,0}=\,{\frac { \overline{a_{{0,2}}}f_{{1,1}}({\omega}-1)^2
-6\,i\alpha f_{{3,0}} \omega+6\,i\alpha{f^{2}_{{2,0}}}
}{3\,{\omega}\left( \omega-1\right)}},\quad a_{2,1}=\frac{i\,
P_{2,1}}{3{\alpha}{ \left( \omega-1 \right) ^{2}}},
$$
with
\begin{align*}
P_{2,1}=&\left(
9\,i\alpha\,f_{{1,1}}f_{{2,0}}-3\,f_{{2,0}}f_{{1,1}}{\alpha}^{
2}+3\,i\alpha\,f_{{2,1}}-2\, \left(  \left| a_{{0,2}} \right|
 \right) ^{2} \right) {\omega}^{2}\\
  &+ \left( 3\,i \left(  \left| f_{{1,1
}} \right|  \right) ^{2}\alpha+3\, \left(  \left| f_{{1,1}} \right|
 \right) ^{2}{\alpha}^{2}
 -3\,i\alpha\,f_{{1,1}}f_{{2,0}}+3\,i\alpha\,f
_{{2,1}}+4\, \left(  \left| a_{{0,2}} \right|  \right) ^{2} \right)
\omega\\
&-3\,i \left(  \left| f_{{1,1}} \right|  \right) ^{2}\alpha-6\,i
\alpha\,f_{{1,1}}f_{{2,0}}-6\,i\alpha\,f_{{2,1}}-2\, \left(  \left|
a_ {{0,2}} \right|  \right) ^{2} ;
\end{align*}
 $$ a_{1,2}={\frac {i\, P_{1,2} }{3\,\left( 1-\omega\right)
 }},
$$
with
\begin{align*}
P_{1,2}=&  \left(
-6\,if_{{1,2}}\alpha+2\,\overline{f_{{1,1}}}a_{{0,2}}+4\,f_{{
2,0}}a_{{0,2}} \right) {\omega}^{2}+i \left(
3\,i\alpha\,{f_{{1,1}}}^{
2}-4\,\overline{f_{{1,1}}}a_{{0,2}}-2\,f_{{2,0}}a_{{0,2}} \right)
\omega\\
& +i \left( 3\,i\overline{f_{{2,0}}}\alpha\,f_{{1,1}}+2\,
\overline{f_{{1,1}}}a_{{0,2}}-2\,f_{{2,0}}a_{{0,2}} \right);
\end{align*} and
 $$ a_{0,3}={\frac {P_{0,3}}{ 3\,\left( 1-\omega \right)}},
$$  with
\[
P_{0,3}=-\left(6\,a_{{0,2}}\overline{f_{{2,0}}}+f_{{1,1}}a_{{0,2}}
\right) { \omega}^{2}+ \left(
2\,a_{{0,2}}\overline{f_{{2,0}}}-f_{{1,1}}a_{{0,2} } \right)
\omega-12\,if_{{0,3}}\alpha+4\,a_{{0,2}}\overline{f_{{2,0}}}
+2\,f_{{1,1}}a_{{0,2}}.
\]
\end{propo}
\begin{proof} As  mentioned before, and as can be seen in the proof
of Proposition \ref{p:propoquadratics}, when $\mathrm{e}^{i\alpha}$ is
 a third root of the unity, the only compatibility condition that
  appears is the one associated with the coefficient  $a_{0,2}$,
  and it is $f_{0,2}=0$. Assuming now this condition, setting $a_{0,2}$
  as a free parameter and fixing the values of the coefficients $a_{2,0}$ and $a_{1,1}$ as
  the ones in Equation \eqref{e:coefsquadratics}, we proceed to compute  the coefficients of the cubic term.
As in the proof of Proposition~\ref{p:propocubics-nores}, we only
show how to obtain the the coefficient $a_{3,0}$. Indeed, we get:
$$
\dot{\varphi}_{3,0}(t)=i\alpha\,\varphi_{3,0}(t)+a_{3,0}\mathrm{e}^{3i\alpha\,t}
+\frac { \widetilde{Q}_{3,0}(t)}{3\,{\omega}^{2} \left( \omega-1
\right) ^{2}},
$$
where  $$ \widetilde{Q}_{3,0}(t)= -6\,i\alpha{f^{2}_{{2,0}}}{{\rm
e}^{2\,i\alpha\,t}}+ \left( -
\overline{a_{{0,2}}}f_{{1,1}}{\omega}^{2}
+6\,i\alpha\,{f^{2}_{{2,0}}}+\overline{a_{{0,2}}}f_{{1,1}}\right)
{{\rm e}^{3\,i\alpha
\,t}}+\overline{a_{{0,2}}}f_{{1,1}}\left(\omega^2-1\right).
$$
By integrating this equation, imposing Equation \eqref{e:temps1} and
taking into account that $\omega^3=1$, we get that the corresponding
Equation \eqref{e:E} is:
 $$ {\frac {i\left( \omega-1\right)
a_{{3,0}}}{2\alpha}}-{\frac {i\left( \overline{a_{{0,2}}}f_{{1,1}}
\left( {\omega}-1\right)^{2}+6\,i\alpha{f^{2}_{{2,0}}} \right) }{
6\omega\,\alpha}}=f_{{3,0}}.
$$
Thus we get the expression of the coefficient $a_{3,0}$ in the statement. The other expressions are obtained similarly.~\end{proof}

If  $\omega=\mathrm{e}^{i\alpha}$ is a squared root of the unity,
then $\alpha=\pi$ (since $\alpha\neq 0$). In this case the
compatibility conditions \eqref{e:compatibilitat}  are the ones
associated with the coefficients $a_{3,0}$ and $a_{1,2}$ but also
$a_{0,3}$, because $\omega^2=1$ implies $\omega^4=1.$ Proceeding as
in the previous results, we obtain:

\begin{propo}
Set $F(z,\bar{z})=-z+\sum_{j+k=2}^3f_{j,k}z^j\bar{z}^k.$ Then there
exists a cubic vector field satisfying \eqref{e:temps1} for $n=3,$
if and only if
\begin{align*}
f_{{3,0}}&=-\frac{1}{2}\,f_{{1,1}}\overline{f_{{0,2}}}-{f^{2}_{{2,0}}},\\
f_{{1,2}}&=-\frac{1}{2}\,{f^{2}_{{1,1}}}-\frac{1}{2}\,f_{{1,1}}\overline{f_{{0,2}}}-f_{
{2,0}}f_{{0,2}}-f_{{0,2}}\overline{f_{{1,1}}},\\
f_{{0,3}}&=-\frac{1}{2}\,f_{{0
,2}}\, \left( 2\,\overline{f_{{2,0}}}+f_{{1,1}} \right)
.
\end{align*}
If these equations are fulfilled, then there is a  three-parameter
family of cubic vector fields satisfying \eqref{e:temps1} for $n=3,$
and it is given by
$$
a_{{2,0}}=\frac{\pi i}{2}\,f_{{2,0}},\quad a_{{1,1}}=-\frac{\pi i}{2}\,f_{{1,1}},\quad a_{{0,2}}=-\frac{3\pi i}{2}\,f_{{0,2}},
$$
$$
a_{{2,1}}=\frac{1}{4}\, \left( i\pi-2 \right)  \left| f_{{1,1}}
 \right|^{2}+\frac{1}{2}\, \left( 3\,i\pi-2 \right)   \left|
f_{{0,2}} \right|   ^{2}
-\left(\frac{3}{2}+\frac{\pi i}{4}\right)f_{{2,0}}f_{{1,1}}
-f_{{2,1}},
$$
and $a_{3,0},a_{1,2}$ and $a_{0,3}$ are free parameters.\end{propo}

The resonant case that appears when   $\omega=\mathrm{e}^{i\alpha}$ is a primitive fourth root of the unity is studied in the following result:

\begin{propo}\label{p:propocubics-resquarta}
Set $F(z,\bar{z})=\omega z+\sum_{j+k=2}^3f_{j,k}z^j\bar{z}^k.$ Then,
\begin{enumerate}[(a)]
\item If $\omega=i$ (that is $\alpha=\frac{\pi}{2}$), then there exists a cubic vector field satisfying \eqref{e:temps1} for
$n=3,$
  if and only if
\begin{equation}\label{e:compatibilitatquartica}
 f_{{0,3}}=\frac{1}{2}\,f_{{0,2}}\, \left(  \left( 2+2\,i \right) \overline{f_{{2,0}}}+
 \left( 1-i \right) f_{{1,1}} \right).
\end{equation} If this equation is fulfilled, then there is an one-parameter family of
vector fields satisfying \eqref{e:temps1} for $n=3,$ given by
$$
a_{{2,0}}=-\frac{\pi}{4}\left( 1+\,i \right)f_{{2,0}},\quad
a_{{1,1}}= \frac{\pi}{4} \left( 1-i \right) f_{{1,1}},\quad
a_{{0,2}}= \frac{3\pi}{4}\left( 1+\,i \right) f_{{0,2}}
$$ and
\begin{align*}
a_{{3,0}}=&-\frac{\pi}{2}\, \left( - \left( 1+\frac{i}{2} \right) f_{{1,1}}
\overline{f_{{0,2}}}+i{f^{2}_{{2,0}}}+f_{{3,0}} \right),\\
a_{{2,1}}=&\frac{1}{4}\, \left( -2+(\pi-2)\,\,i \right)    \left| f_{{1,1}}
 \right|  ^{2}+\frac{1}{2}\, \left( -2+(3\pi+2)\, i\right)
 \left| f_{{0,2}} \right|   ^{2}\\
&+\frac{1}{4}\,\left( 6+(\pi-2)
\,i \right) f_{{1,1}}f_{{2,0}} -if_{{2,1}},\\
a_{{1,2}}=&\frac{\pi}{2}\, \left( - \left( 2+i \right) f_{{0,2}}
\overline{f_{{1,1}}}-\frac{i}{2}f_{{1,1}}\overline{f_{{2,0}}}- \left( 2-i
 \right) f_{{2,0}}f_{{0,2}}+\frac{i}{2}{f^{2}_{{1,1}}}+f_{{1,2}}
 \right),
\end{align*}
being $a_{0,3}$ the free parameter.
\item If $\omega=-i$ (that is $\alpha=\frac{3\pi}{2}$), then there exists
 a cubic vector field satisfying \eqref{e:temps1} for $n=3,$ if and only if
$$
f_{{0,3}}=\frac{1}{2}\,f_{{0,2}}\, \left(  \left( 1+i \right) f_{{1,1}}+ \left( 2-2\,i
 \right) \overline{f_{{2,0}}} \right).
$$
If this equation is fulfilled, then the there is an one-parameter
family of cubic vector fields  satisfying \eqref{e:temps1} for
$n=3,$ and it is given by
$$a_{{2,0}}= \frac{3\pi}{4}\left( 1-\,i
 \right) f_{{2,0}},\quad
a_{{1,1}}=
 -\frac{3\pi}{4}\left( 1+\,i \right) f_{{1,1}},\quad
a_{{0,2}}=\frac{9\pi}{4} \left( -1+\,i \right) f_{{0,2}}
$$ and
\begin{align*}
a_{{3,0}}=&-\frac{3\pi}{2}\, \left(  \left( 1-\frac{i}{2} \right) f_{{1,1}}
\overline{f_{{0,2}}}+i{f^{2}_{{2,0}}}-f_{{3,0}} \right),\\
a_{{2,1}}=&\frac{1}{4}\, \left( -2+(3\pi+2)  \,i \right)   \left| f_{{1,1
}} \right|   ^{2}+\frac{1}{2}\, \left( -2+(9\pi-2)\,i \right)
   \left| f_{{0,2}} \right|   ^{2}\\
   &+\left(\frac{3}{2}+i \left( \frac{1}{2}+\frac{3}{4}\,\pi \right) \right)
    f_{{2,0}}f_{{1,1}}+if_{{2,1}},\\
a_{{1,2}}=&\frac{3\pi}{2}\, \left(  \left( 2-i \right) f_{{0,2}}
\overline{f_{{1,1}}}-\frac{i}{2}\overline{f_{{2,0}}}f_{{1,1}}+ \left( 2+i
 \right) f_{{2,0}}f_{{0,2}}+\frac{i}{2}{f_{{1,1}}}^{2}-f_{{1,2}}
 \right),
\end{align*}
being $a_{0,3}$ the free parameter.
\end{enumerate}
\end{propo}

\section{Proof of Theorem \ref{t:main}}\label{s:proof}

 Theorem \ref{t:main} is a consequence of the following two results.
  The first one is proved in \cite{CGM18} but  for completeness we include
   a sketch of its proof. The second one is a consequence of the results in the previous section.

\begin{propo}\label{p:maps} The two polynomial maps
$$
F_1(z,\bar{z})=iz+(1-3i)z^2+z\bar{z}\quad\mbox{and} \quad F_2(z,\bar{z})=\frac{1}{2}\left(1+i\sqrt{3}\right)z-z^2\bar{z},
$$
have the origin as  a LAS
 fixed point for both of them, while the composition map $F_2\circ F_1$ has the origin as a repeller fixed point.
\end{propo}

\begin{proof}
Let $\mathcal{U}$ be a small enough neighborhood of the origin. A
$\mathcal{C}^{2m+2}$ map $F$ in $\mathcal{U}$ with an elliptic fixed
point whose eigenvalues $\lambda,\bar \lambda=1/\lambda,$  are not
roots of unity of order $\ell$ for $0<\ell\leq 2m+1,$  is locally
conjugate to its \emph{{Birkhoff normal form}}:\[ F_{B}(z,\bar
z)=\lambda z\Big(1+\sum_{j=1}^{m} B_j
(z\bar{z})^j\Big)+O(2m+2),\]
see \cite{AP}. The first non-vanishing
number $B_j$ is called the $j$th \emph{Birkhoff
    constant}.
If $V_j=\mathrm{Re}(B_j)< 0$ (resp. $V_j>0$), then the point $p$ is
LAS (resp. repeller), see  \cite[Lem. 4.1]{CGM18} for instance.  The
quantity $V_j$ is called  the $j$th \emph{Birkhoff    stability
constant}. This is so, because the fact that $V_j\ne 0$ implies that
the function $z\overline z = |z|^2$  is a strict Lyapunov function
at the origin for the normal form map $F_B$ of $F.$

In~\cite{CGM18}, both  the Birkhoff and the Birkhoff
stability constants of $F_1$ and $F_2$ are computed obtaining that
$
B_1(F_1)=-\frac{1}{2}-\frac{11}{2}\, i$  and
$B_1(F_2)=-\frac{1}{2}+\frac{\sqrt{3}}{2}\, i.
$
So  $V_1(F_j)=-\frac{1}{2}<0$ for $j=1,2$, and the origin is LAS for both maps $F_1$ and $F_2$. Also in this reference it is proved that
$V_1(F_2\circ F_1)=\frac{1}{2}\left(3\sqrt{3}-5\right)>0$, so that the origin is a repeller fixed point for the composition map.~\end{proof}

\begin{propo}\label{p:laslasrep}
    Consider the planar polynomial vector fields
    \begin{align*}
    X_{1}(z,\bar{z},\mu)=&\frac{i\pi}{2} z- \left( 1-\frac{i}{2} \right) \pi {z}^{2}+
    \left( \frac{1}{4}-\frac{i}{4} \right) \pi z\bar{z}
    - \left( 3-4 i \right) \pi {z}^{3}\\
    & +\left( \frac{3\pi}{4} -\frac{1}{2}+i \left( \frac{\pi}{2}-\frac{11}{2} \right)  \right) {z}^{2}\bar{z}
    +\frac{3\pi}{4}  z\bar{z}^{2}+\mu\,\bar{z}^{3},
    \end{align*}
    where $\mu$ is a free parameter, and
    $$
    X_2(z,\bar{z})=\frac{i\pi}{3} z+ \left( -\frac{1}{2}+\frac{i\sqrt {3}}{2} \right) {z}^{2}\bar{z}.
    $$
    Let $\varphi_j(t;z,\bar{z}), j=1,2$ be their respective associated flows. Then, for $z$ in a small enough neighborhood of the origin
    \[
\varphi_j(1;z,\bar{z})=F_j(z,\bar{z})+O(4),\quad j=1,2,
    \]
where the maps  $F_j$ are given in Proposition~\ref{p:maps}.
\end{propo}

\begin{proof}
    Observe that $F_1$ has the form \eqref{e:F} with $\alpha=\pi/2$, so
    that $\mathrm{e}^{i\alpha}$ is a primitive fourth root of the unity.
    Since the compatibility condition \eqref{e:compatibilitatquartica}
    is satisfied, by using the expression in Proposition
    \ref{p:propocubics-resquarta}$(a)$ we can find an one-parameter family
    of vector fields $X_{1}(z,\bar{z},\mu)$ satisfying
    $\varphi_1(1;z,\bar{z})=F_1(z,\bar{z})+O(4).$ This is the family of vector fields
    $X_{1}$ given in the statement,
    where $\mu$ is the free parameter $a_{0,3}$. Also observe that $F_2$
    has also the form \eqref{e:F} with $\alpha=\pi/3$, so that
    $\mathrm{e}^{i\alpha}$ is a primitive sixth root of the unity. By
    using Proposition \ref{p:propocubics-nores} we can find a unique
    vector field $X_2$ satisfying
    $\varphi_2(z,\bar{z})=F_2(z,\bar{z})+O(4)$. This $X_2$ is the second vector field
    given in the statement.
\end{proof}

\begin{proof}[Proof of Theorem \ref{t:main}]  We will prove that the vector
 fields given in the statement of Proposition~\ref{p:laslasrep} provide the
 desired example with $X_1$ and $X_2$ having the origin as a singular LAS point
  and with the origin being a repeller for the 2-seasonal differential system~\eqref{e:eq1}.
   Then, the converse situation will hold  simply by considering the vector fields  $-X_1$ and $-X_2.$

The key point is to realize that if $\varphi(t;z,\bar z)$ denotes the flow of~\eqref{e:eq1} it holds that
\begin{align*}
\varphi(2;z,\bar{z})&=\varphi_2(1;\varphi_1(1;z,\bar{z}),\overline{\varphi_1(1;z,\bar{z})} )
=F_2\big(F_1(z,\bar{z}\big)+O(4))+O(4)\\&=F_2\circ F_1(z,\bar{z})+O(4).
\end{align*}
Now, a crucial step is that the first Birkhoff
stability constant $V_1(F)$ only depends on the third order jet of $F$ at the fixed point, see \cite[Equation (3)]{CGM18}.
 Hence $V_1(\varphi_j(1;z,\bar z))=V_1(F_j),$  $j=1,2$ and
 $V_1(\varphi(2;z,\bar z))=V_1(F_2\circ F_1).$

It is clear that for the vector fields $X_1$, $X_2$ and the one in \eqref{e:eq1} the stability of
the origin coincides  with the one of the corresponding flows
$\varphi_1(1;\cdot,\cdot),$  $\varphi_2(1;\cdot,\cdot)$ and
$\varphi(2;\cdot,\cdot)$ respectively. Equivalently, these stabilities coincide
with the ones of the origin for the maps $F_1,$ $F_2$ and $F_2\circ
F_1.$ Since, by Proposition~\ref{p:maps}, these maps provide a
discrete dynamic Parrondo's paradox, we  have that both $X_1$ and
$X_2$ have a LAS singular point at the origin, and the corresponding
$2$-seasonal system~\eqref{e:eq1} has a repeller point at the
origin, as we wanted to prove. ~\end{proof}

\section{Proof of Theorem \ref{p:minimal}}\label{t:2}

$(i)$ This is a corollary of statement $(i)$  of Theorem \ref{p:compatibilitat}.

$(ii)$ We will use item $(ii)$  of Theorem \ref{p:compatibilitat}.
For each $n\ge 2$ we will prove that the polynomial map
\begin{equation}\label{e:Fcontaexemple}
F(z,\bar{z})=\mathrm{e}^{i\alpha} z+\bar{z}^n,\,\,\text{with}\,\,\alpha=\frac{2\,\pi}{n+1},
\end{equation}
satisfies the statement of the theorem.

The result for $n=2$ is a direct consequence of
Proposition~\ref{p:propoquadratics}. When   $n=3,$ the result
follows by item $(a)$ of Proposition~\ref{p:propocubics-resquarta}
because  the compatibility
condition~\eqref{e:compatibilitatquartica} does not hold.

Now suppose that $n\ge 4.$  We claim that  \emph{for each $2\le m\le
n-1,$ if  $\mathrm{e}^{i\alpha}$ is a primitive $(n+1)$-root of
unity and  $X_m$ is a vector field with associated flow of the form
 $\varphi_m(t;z,\bar z)= \mathrm{e}^{i\alpha t} z+O(m+1),$ then  it satisfies $X_m(z,\bar z)= i\alpha z
+O(m+1).$}  We will prove the claim by induction on $m$, by using
the same method and notations introduced in
Section \ref{s:recu}.

By Proposition \ref{p:propoquadratics} the result is true for $m=2.$
Assume that the result is true for $m<n-1.$ As a consequence, for any vector field of the form
$$X_{m+1}(z,\bar{z})=i\alpha z+\sum_{j+k=m+1} a_{j,k}
z^j\bar{z}^k+ O(m+2),$$  its associated  flow has the form
$$\varphi_{m+1}(t;z,\bar{z})=\mathrm{e}^{i\alpha t} z+\sum_{j+k=m+1}\varphi_{j,k}(t)z^j\bar{z}^k+O(m+2). $$
By plugging the above expression into the
differential system $\dot{z}=X_{m+1}(z,\bar{z})$, we get that for $j+k=m+1$:
\begin{equation}\label{eq1exemple}
\varphi_{j,k}'(t)=i\alpha
\varphi_{j,k}(t)+a_{j,k}\mathrm{e}^{(j-k)i\alpha t},
\end{equation}
and since $\varphi_{j,k}(0)=0$ we obtain that
\begin{equation}\label{eq2exemple}
\varphi_{j,k}(t)=\begin{cases}
\dfrac{a_{j,k}}{(j-k-1)i\alpha}\,\mathrm{e}^{i\alpha t}\,\left(\mathrm{e}^{(j-k-1)i\alpha t}-1\right), &j \ne  k+1,\\
a_{j,k} t \,\mathrm{e}^{i\alpha t}, &j  =k+1.
\end{cases}
\end{equation}
Since $\mathrm{e}^{i\alpha}$ is a primitive $(n+1)$-root of unity,
 $\mathrm{e}^{(j-k-1)i\alpha}\ne 1$ for $j+k=m+1<n$ (see Remark \ref{resnivelln}). Now, if we assume that $\varphi_{m+1}(1;z,\bar{z})=\mathrm{e}^{i\alpha} z+O(m+1)$
then  $\varphi_{j,k}(1)=0$ for $j+k=m+1$ and from
(\ref{eq2exemple}) $a_{j,k}=0$ and $\varphi_{j,k}(t)\equiv0.$ So,
the claim is proved.

Now we proceed by contradiction. We consider the map \eqref{e:Fcontaexemple}. and suppose that there exists a vector field $X$ whose flow satisfies
\begin{equation}\label{e:vphiigualaF}
\varphi(1;z,\bar{z})=F(z,\bar{z})+O(n+1)=\mathrm{e}^{i\alpha} z+\bar{z}^n+O(n+1).
\end{equation}
From the claim, $X$ must have the form
$X(z,\bar{z})=i\alpha z+\sum_{j+k=
n} a_{j,k}z^j\bar{z}^k+O(n+1)$. For these kind of vector fields the associated flow must have the form
$\varphi(t;z,\bar{z}) =\mathrm{e}^{i\alpha t}+\sum_{j+k= n}
\varphi_{j,k}(t) z^j\bar{z}^k+O(n+1).$ For $j+k=n$, the
functions $\varphi_{j,k}(t)$ also satisfy (\ref{eq1exemple}) and
hence (\ref{eq2exemple}). In particular,
$$
\varphi_{0,n}(t)=
-\dfrac{1}{2\pi i}\,a_{0,n}\,\mathrm{e}^{\frac{2\pi i }{n+1}\, t}\,
\left(\mathrm{e}^{-2\pi i t}-1\right),
$$
and therefore $\varphi_{0,n}(1)=0$.  But this is in contradiction
with  Equation \eqref{e:vphiigualaF}, which implies that $\varphi_{0,n}(1)=1$.

\end{document}